\documentclass[12pt]{amsart}
\usepackage{amssymb}
\usepackage{amsmath,amsfonts,amssymb,upgreek}

\usepackage[breaklinks]{hyperref}
\usepackage{amscd}
\usepackage{mathtools}

\theoremstyle{thmrm}
\usepackage{tikz}
\usepackage{enumerate}
\usetikzlibrary{calc,decorations.markings}
\usepackage{tikz-cd}
\usepackage{graphicx}

\theoremstyle{plain}
\newtheorem{no}{Notation}[section]

\newtheorem{thm}{Theorem}[section]
\newtheorem{lemma}[thm]{Lemma}
\newtheorem{prop}[thm]{Proposition}

\newtheorem{defn}[thm]{Definition}
\newtheorem{remark}[thm]{Remark}
\newtheorem{example}[thm]{Example}

\newcommand{\cd}{\operatorname{cd} }

\newcommand{\op}{\operatorname{op} }
\newcommand{\GL}{\operatorname{GL} }

\newcommand{\Cl}{\operatorname{Cl} }

\newcommand{\Aut}{\operatorname{Aut} }

\newcommand{\Irr}{\operatorname{Irr} }

\numberwithin{equation}{section}

\begin{document}
	
	\title{Representations of Skew Left Braces of order $pq$}
	
	\author{Nishant Rathee}
	\address{Stat-Math Unit, Indian Statistical Institute, 7 S. J. S. Sansanwal Marg, New Delhi, 110016, India} 
	\email{monurathee2@gmail.com, nishant@isid.ac.in}

	\author{Ayush Udeep$^*$}
	\address{Department of Mathematical Sciences, Indian Institute of Science Education and Research (IISER) Mohali, Sector 81, SAS Nagar, P O Manauli, Punjab 140306, India} 
	\email{udeepayush@gmail.com}

	
	\begin{abstract}
In this paper, we study the irreducible representations of skew braces of order \( pq \), which is equivalent to studying the representation theory of groups of order \( p^2q^2 \) arising from skew left braces, where \( p > q \) are primes. To achieve this, we classify all semidirect product groups \( \Lambda_A \) associated with skew left braces $A$ of order \( pq \), up to isomorphism.

\end{abstract}
	\maketitle

\section{Introduction}

Rump introduced the concept of left braces \cite{WR07}, providing an algebraic framework for constructing and classifying involutive, non-degenerate set-theoretic solutions of the Yang–Baxter equation. Guarnieri and Vendramin \cite{GV17} later extended this notion to skew left braces, a non-abelian generalization that yields a broader class of bijective, non-degenerate solutions. The study of skew left braces has since expanded in various directions, with significant focus on their algebraic properties, including nilpotency, extension theory, and their deep connections to Hopf-Galois theory \cite{DB18, CSV18, AV18}. These developments establish skew left braces as a central object of study, bridging multiple areas of algebra and mathematical physics. 

In recent years, the structural properties of skew left braces have been extensively studied (see \cite{LV23, LV24}).
A skew left brace is a set \(A\) equipped with two group structures, \((A, \cdot)\) and \((A, \circ)\), related by the identity 
\[
a\circ (b \cdot c) = (a \circ b) \cdot a^{-1} \cdot (a \circ c),\]
for all \( a,b,c \in A\), where \(a^{-1}\) denotes the inverse in the group \((A, \cdot)\). This structure induces a natural action 
\[
\lambda: (A, \circ) \to \operatorname{Aut} (A, \cdot), \quad \lambda_a(b) = a^{-1} \cdot (a \circ b),
\]
which is known as the lambda map associated with skew left brace $A$. This action allows one to construct the associated semidirect product group 
\[
\Lambda_A = (A, \cdot) \rtimes_{\lambda} (A, \circ).
\]

Letourmy and Vendramin \cite{LV24} defined a representation of a skew brace as a pair of representations on the same vector space, one for the additive group and the other for the multiplicative group, satisfying a certain compatibility condition (also see \cite{KT24}).  
It was observed in \cite{RSU24} that there exists a one-to-one correspondence between representations of a skew brace $A$ and those of $\Lambda_A$, marking a crucial breakthrough in the study of skew brace representations.
The group \(\Lambda_A\) has been studied and referred to as the triply factorized group in \cite{BR21}, where various structural properties, including left and right nilpotency, substructures, and Fitting-like ideals, were explored. These properties play a fundamental role in understanding the structure of skew left braces and their associated group representations. 

 Acri and Bonatto \cite{AB2020} classified all skew left braces of size $pq$, where \(p > q\) are prime numbers (see Theorem \ref{classification}). In the present paper, we classify the irreducible representations of all skew left braces of size \(pq\).

\begin{thm} \label{thm:mainresult}
	Let $A$ be a skew left brace of size $pq$. If $p \not\equiv 1 \bmod q$, then there are exactly $p^2 q^2$ many 1-dimensional representations. If $p \equiv 1 \bmod q$, then 
	\begin{enumerate}
		\item $A = B$ has exactly $p^2 q^2$ many 1-dimensional representations.
		\item $A = C$ has exactly $q^2$ many 1-dimensional representations and $p^2 - 1$ many $q$-dimensional irreducible representations of $A$.
		\item if $A = D$, or $F_{\gamma}$ for $1< \gamma < q$, or $G_{\mu}$ for $1< \mu \leq q$, then $A$ has exactly $q^2$ many 1-dimensional representations, $2(p-1)$ many $q$-dimensional irreducible representations and $(p-1)^2/q^2$ many $q^2$-dimensional irreducible representations.
		\item if $A = E$, or $F_q$, then $A$ has exactly $pq^2$ many 1-dimensional representations and $p^2 - p$ many $q$-dimensional irreducible representations.
	\end{enumerate}
	Here $B,C,D, E, F_\gamma$ and $G_\mu$ are the same as in Theorem \ref{classification}.
\end{thm}
\noindent To prove Theorem \ref{thm:mainresult}, we obtain the irreducible representations of the associated group $\Lambda_{A}$ of order $p^2q^2$ and utilize the isomorphism of such groups. We also classify the groups $\Lambda_{A}$ associated with the skew left braces $A$ of size $pq$, up to isomorphism, in Theorem \ref{thm:classification}.

 To the best of our knowledge, representations of groups of order $p^2q^2$ have not been computed yet, owing to the fact that the representation theory of arbitrary groups of order \(p^2q^2\) is generally intractable.
Our results demonstrate how skew brace theory provides an effective framework for addressing problems in representation theory that are otherwise difficult to approach purely from a group-theoretic perspective.
 A key observation in our approach is that the commutator subgroup of \(\Lambda_A\) can be computed in terms of the commutator of the corresponding skew brace, a relation that was established in \cite{RSU24}.

We mention the preliminaries required in our computations in Section \ref{sec:prelims}. In Section \ref{sec:results}, we prove Theorem \ref{thm:mainresult} and \ref{thm:classification}. We conclude the article with examples in which we compute the irreducible representations of a skew brace and its associated group (see Example \ref{example: q dim rep skew brace} and \ref{example:q dim rep group theory}).

\section{Preliminaries} \label{sec:prelims}

Let $(A,\cdot ,\circ)$ be a skew left brace. The groups $(A,\cdot)$ and $(A, \circ)$ are called the {\it additive} and the {\it multiplicative} groups of $(A,\cdot ,\circ)$, respectively. The structure $(A, \cdot^{\op}, \circ)$ is also a skew left brace, called the {\it opposite} skew left brace of $A$, where $a \cdot^{\op} b := b \cdot a$ for all $a, b \in A$.
Note that the lambda map  $\lambda^{\op}$ of $(A, \cdot^{\op}, \circ)$ is a group homomorphism given by
$$
\lambda^{\op}:(A, \circ)\to \Aut(A, \cdot), \lambda^{\op}_a(b) = (a \circ b) \cdot a^{-1} = a \cdot \lambda_a(b) \cdot a^{-1}.
$$
 A subset $I$ of $A$ is called an {\it ideal} of $(A, \cdot, \circ)$ if it is a normal subgroup of both $(A,\cdot)$ and $(A,\circ)$, and $\lambda_a(I)\subseteq I$ for all $a\in A$. The {\it commutator} $A'$ of $(A, \cdot, \circ)$ is a subgroup of $(A, \cdot)$ generated by elements of the form $a \cdot b\cdot a^{-1} \cdot b^{-1}$ and $a \cdot \lambda_b(a^{-1})$ for all $a, b \in A$. The commutator $A'$ turns out to be an ideal of $(A, \cdot, \circ)$.

\begin{remark}
\textnormal{The commutator of a skew brace is the smallest ideal that makes the quotient a trivial brace.}
\end{remark}

\begin{no} 
	\textnormal{Before proceeding further, we set some notations:
	\begin{itemize}
		\item For groups $G$ and $H$ and a homomorphism $\psi: G \to \Aut(H)$, the group operation in the semi-direct product $H \rtimes_{\psi} G$ is given by 
		$$(h_1, g_1)(h_2, g_2)=(h_1 \psi_{g_1}(h_2), g_1 g_2)$$
		for all $h_1, h_2 \in H$ and $g_1, g_2 \in G$. 
		\item For a skew left brace $A:=(A, \cdot, \circ)$,   
		$$\Lambda_{A} = (A,\cdot)\rtimes_{\lambda}(A,\circ) \quad \textrm{and} \quad \Lambda_{A^{\op}} = (A,\cdot)\rtimes_{\lambda^{\op}}(A,\circ)$$ are the semi-direct products induced by $\lambda$ and $\lambda^{\op}$, respectively.
		\item For a group $G$ and an element $g$ of $G$, $G'$ denotes the commutator subgroup of $G$, $\Cl^G(g)$ denotes the conjugacy class of $g$ in $G$ and $k(G)$ denotes the number of conjugacy classes of $G$.
	\end{itemize}}
\end{no}

Next, we state a fundamental result concerning the isomorphism of semidirect product groups.

\begin{thm}\label{semi isomorphism}
Let $G$ and $H$ be groups, and let $\psi: G \to \Aut(H)$ be a homomorphism. For any $f \in \Aut(H)$, the semi-direct product groups $H \rtimes_{\psi f} G$ and $H \rtimes_{\psi} G$ are isomorphic under the map  
		\[
		\varphi: H \rtimes_{\psi f} G \to H \rtimes_{\psi} G, \quad (g,h) \mapsto (g, f(h)),
		\]
		for all $g \in G$ and $h \in H$.
\end{thm}

\begin{lemma} \label{prop:AcongAop} \label{commutator1}
	Let $A$ be a skew left brace. Then
	\begin{enumerate}
		\item [\rmfamily(i)] \textnormal{(\cite[Proposition 3.3]{RSU24})} $\Lambda_{A} \cong \Lambda_{A^{\op}}$, under the map $(x,y) \mapsto (xy,y)$.
		\item [\rmfamily(ii)] \textnormal{(\cite[Theorem 3.7]{RSU24})} $\Lambda_{A}'=A' \rtimes_{\lambda} (A, \circ)'$.  
	\end{enumerate}
\end{lemma}

In  \cite{LV23}, Letourmy and Vendramin defined a representation of a skew left brace $A$ as follows.
\begin{defn} A representation of a skew left brace $A$ is a triple $(V,\beta,\rho)$, where
	\begin{enumerate}[(1)]
		\item $V$ is a vector space over some field;
		\item $\beta : (A,\cdot)\longrightarrow \GL(V)$ is a representation of $(A,\cdot)$;
		\item $\rho: (A,\circ)\longrightarrow \GL(V)$ is a representation of $(A,\circ)$;
	\end{enumerate}
	such that the relation
	\begin{equation}\label{relation}
		\beta(\lambda^{\op}_a(b)) = \rho(a)\beta(b)\rho(a)^{-1}
	\end{equation}
	holds for all $a,b\in A$
	\end{defn}

\begin{lemma} \textnormal{\cite[Corollary 4.7]{RSU24}} \label{one one corres rep}
	There is a one-to-one correspondence between the set of equivalence classes of irreducible representations of a skew left brace $A$ and the set of equivalence classes of irreducible representations of the group $\Lambda_{A^{\op}}$.
\end{lemma}

\begin{remark} \label{remark:equivalence of reps}
\textnormal{By Lemma \ref{prop:AcongAop} and \ref{one one corres rep}, it is enough to classify the irreducible representations of $\Lambda_A$ to classify the irreducible representations of a skew left brace $A$.} 
\end{remark}
	
	\section{Results} \label{sec:results}
	In this section, we classify the groups $\Lambda_{A}$ for skew left braces $A$ of size $pq$, up to isomorphism. With the help of that, we classify the irreducible representations of all the skew left braces of size $pq$. We conclude the section with an example of the computation of an irreducible representation for one of the skew left braces. Throughout this section, $p$ and $q$ denote prime numbers such that $p > q$, and all the groups are finite.
	
	In Section \ref{subsec:classification}, we obtain results on the structural properties of the groups associated with skew left braces of size $pq$.
	\subsection{Groups arising from skew left braces of size $pq$} \label{subsec:classification}
	
		The following classification of skew left braces of size $pq$ is given in \cite{AB2020}.
	
		\begin{thm} \textnormal{\cite[Theorem, p. 1873]{AB2020}} \label{classification}
		 If $p \not\equiv 1 \pmod{q}$, there is only one skew brace of order $pq$, the trivial one. If $p \equiv 1 \pmod{q}$, a complete list of the $2q + 2$ skew left braces of order $pq$, up to isomorphism, is given below, where $g$ is a fixed element of $\mathbb{Z}_p$ of multiplicative order $q$:
		
		\begin{itemize}
			\item \textbf{Additive group} $\mathbb{Z}_p \times \mathbb{Z}_q$:  
			\[
			(n, m) + (s, t) = (n + s, m + t).
			\]
			\begin{enumerate}
				\item The trivial skew brace, denoted by $B$
				\item The bi-skew brace, denoted by $C$ where
				\[
				(n, m) \circ (s, t) = (n + g^m \cdot s, m + t).
				\]
			\end{enumerate}
			
			\item \textbf{Additive group} $\mathbb{Z}_p \rtimes \mathbb{Z}_q$:  
			\[
			(n, m) + (s, t) = (n + g^m \cdot s, m + t).
			\]
			\begin{enumerate}
				\item The trivial skew brace, denoted by $D$.
				\item The skew brace, denoted by $E$, where
				\[
				(n, m) \circ (s, t) = (g^t \cdot n + g^m \cdot s, m + t).
				\]
				\item The bi-skew braces, denoted by $F_\gamma$ for $1 < \gamma \leq q$, where
				\[
				(n, m) \circ (s, t) = (n + (g^{\gamma})^m \cdot s, m + t).
				\]
				\item The skew left braces, denoted by $G_\mu$ for $1 < \mu \leq q$, where
				\[
				(n, m) \circ (s, t) = (g^t \cdot n + (g^{\mu})^m \cdot s, m + t).
				\]
			\end{enumerate}
		\end{itemize}
		
	\end{thm}
Next, we establish an important result concerning the commutator of a skew brace of order \( pq \).

	\begin{lemma}\label{lemma:commu}
		Let $A$ be a skew left brace of order $pq$. Then, $|A'| = p$ unless $A$ is a trivial brace.
	\end{lemma}
	
	\begin{proof}
		Since $p > q$, the additive group of $A$ has a unique subgroup of order $p$, which we denote by $P$. It is straightforward to verify that $P$ is characteristic and thus invariant under the lambda map, making it a normal subgroup of the multiplicative group and hence an ideal.  Moreover, since $A/P$ has order $q$, it must be a trivial brace. Now, suppose that $A$ is not a trivial brace. Since we know that $|A^\prime| > 1$ and that $A^\prime \subset P$, it follows that $A^\prime = P$.
	\end{proof}
	
	Now, we obtain the structural properties of the groups arising from the skew left braces of size $pq$. These results help us in classifying such groups in Theorem \ref{thm:classification}.
	We shall use the notations of Theorem \ref{classification}.
	
	If $p \not\equiv 1 \bmod q$, then there is only the trivial skew brace, say $A$. In this case, $\Lambda_{A} = \mathbb{Z}_{pq} \times \mathbb{Z}_{pq}$, which is an abelian group of order $p^2 q^2$. 
	
		Now, assume that $p \equiv 1 \bmod q$. By Theorem \ref{classification}, the skew left braces are: $B$, $C$, $D$, $E$, $F_{\lambda}$ for $1< \lambda \leq q$, and $G_{\mu}$ for $1 < \mu \leq q$.
	We deal with each skew left brace separately.\\
	
\noindent	{\bf Skew left brace $B$:} Since $B$ is a trivial skew brace, $\Lambda_{ B} = \mathbb{Z}_{pq} \times \mathbb{Z}_{pq}$, which is an abelian group of order $p^2 q^2$.\\

\noindent	{\bf Skew left brace $ C$:} By Lemma \ref{lemma:commu}, we obtain $| C^\prime| = p$. Moreover, it is easy to see that $( C, \circ)$ is non-abelian. Hence, from Lemma \ref{commutator1}, we get $|\Lambda_{ C}^\prime| = p^2$.
Thus, $\Lambda_{ C}$ is a non-abelian group of order $p^2q^2$.	The associated lambda map of $C$ is given by  
\begin{equation} \label{eq:lambda of A2}
		\lambda_{(n,m)}(s, t) = (-n,-m) + ((n,m) \circ (s,t)) =  (g^{m} \cdot s, t), 
\end{equation}
$\text{ for all } n, s \in \mathbb{Z}_p \text{ and } m, t \in \mathbb{Z}_q$. Now, let $P$ and $Q$ denote the subgroups $\mathbb{Z}_{p} \times \{ 0 \} \times \mathbb{Z}_{p} \times \{ 0 \} \cong \mathbb{Z}_{p} \times \mathbb{Z}_{p}$ and $\{ 0 \} \times \mathbb{Z}_{q} \times \{ 0 \} \times \mathbb{Z}_{q} \cong \mathbb{Z}_{q} \times \mathbb{Z}_{q}$ of $\Lambda_{C}$, respectively. Take $(x,y) \in Q$ and $(a,b) \in P$. Then 
\begin{align*}
	& ((0,x),(0,y)) \ast ((a,0), (b,0)) \ast ((0, x), (0,y))^{-1} \\
	& = ((0,x), \lambda_{(0,y)}(a,0), (b,0)) \ast ((0, -x), (0, -y))
	= ((g^{y}\cdot a, x), (g^{y}\cdot b, y)) \ast ((0, -x), (0, -y)) \\
	& = ((g^{y}\cdot a, 0), (g^{y}\cdot b, 0)).
\end{align*}
	Hence,  
	$
	\Lambda_{C} \cong (\mathbb{Z}_{p} \times \mathbb{Z}_{p}) \rtimes_{\theta} (\mathbb{Z}_{q} \times \mathbb{Z}_{q}),
	$  
	where 
	\begin{equation} \label{eq:theta of C}
	\theta:  \mathbb{Z}_q \times \mathbb{Z}_q \rightarrow \Aut(\mathbb{Z}_p \times \mathbb{Z}_p), \quad	\theta_{(x,y)}(a,b) = (g^y \cdot a, g^y \cdot b),
	\end{equation}  
	for all \( x, y \in \mathbb{Z}_{q} \) and \( a, b \in \mathbb{Z}_{p} \).\\
	
	\noindent	{\bf Skew left brace $ D$:} Since $ D$ is a trivial skew brace over $\mathbb{Z}_p \rtimes \mathbb{Z}_{q}$, $\Lambda_{ D} = \mathbb{Z}_{p} \rtimes \mathbb{Z}_{q} \times \mathbb{Z}_{p} \rtimes \mathbb{Z}_{q}$. Hence, $|\Lambda_{ D}^\prime| = p^2$. Further, it is easy to see that the associated lambda map of $D$ is given by 
	\[ \lambda_{(n,m)}(s, t) =  (s, t), \text{ for all } n, s \in \mathbb{Z}_p \text{ and } m, t \in \mathbb{Z}_q. \]
	Direct computations show that 
	$
	\Lambda_{ D} \cong (\mathbb{Z}_{p} \times \mathbb{Z}_{p}) \rtimes_{\psi} (\mathbb{Z}_{q} \times \mathbb{Z}_{q}),
	$  
	where 
	\begin{equation} \label{eq:psi of D}
	 \psi:  \mathbb{Z}_q \times \mathbb{Z}_q \rightarrow \Aut(\mathbb{Z}_p \times \mathbb{Z}_p), \quad	\psi_{(x,y)}(a,b) = (g^x \cdot a, g^y \cdot b),
	\end{equation}  
	for all \( x, y \in \mathbb{Z}_{q} \) and \( a, b \in \mathbb{Z}_{p} \).\\
	
	\noindent {\bf Skew left brace $ E$:} It is easy to verify that $(E, \circ)$ is abelian. By Lemma \ref{commutator1} and \ref{lemma:commu}, it follows that $|\Lambda_{ E}'| = p$. The associated lambda map of $E$ is given by  
	\[
	\lambda_{(n,m)}(s, t) = \left(s + g^{-m}(g^t - 1) \cdot n, t\right).
	\]  
	
	 Direct computations show that  
	\[
	\Lambda_{ E} \cong (\mathbb{Z}_p \times \mathbb{Z}_p) \rtimes_{\xi} (\mathbb{Z}_q \times \mathbb{Z}_q),
	\]  
	where  
	\begin{equation} \label{eq:xi of E}
		\xi:  \mathbb{Z}_q \times \mathbb{Z}_q \rightarrow \Aut(\mathbb{Z}_p \times \mathbb{Z}_p), \quad
		\xi_{(x,y)}(a,b) = (g^x a + b (1 - g^x), b),
	\end{equation} 
	  for all \( x, y \in \mathbb{Z}_{q} \) and \( a, b \in \mathbb{Z}_{p} \).\\
	
	\noindent {\bf Skew left brace $F_{\gamma}$:} The lambda map of $F_{\gamma}$ is given by 
	$$\lambda_{(n,m)}(s,t)=(g^{\gamma m -m} \cdot s, t).$$
	Direct calculations shows that, for $1< \gamma \leq q$, 
	\[
	\Lambda_{F_{\gamma}} \cong (\mathbb{Z}_p \times \mathbb{Z}_p) \rtimes_{\delta_{\gamma}} (\mathbb{Z}_q \times \mathbb{Z}_q),
	\]  
	where  
	\begin{equation} \label{eq:delta of Alambda}
	\delta_{\gamma}: \mathbb{Z}_q \times \mathbb{Z}_q \to \Aut(\mathbb{Z}_p \times \mathbb{Z}_p), \quad	{\delta_{\gamma}}_{(x,y)}(a,b) = (g^{x+y(\gamma - 1)} \cdot a, g^{\gamma y} \cdot b),
	\end{equation} 
	for all \( x, y \in \mathbb{Z}_q \) and \( a, b \in \mathbb{Z}_p \).
	\begin{enumerate}
		\item {\bf When $\gamma = q$:} In this case, $|\Lambda_{F_{q}}'| = p$ since $(F_q, \circ)$ is abelian.  Further, $\Lambda_{F_q} \cong \mathbb{Z}_{p} \times \mathbb{Z}_{p} \rtimes_{\delta_{q}} \mathbb{Z}_{q} \times \mathbb{Z}_{q}$, where the action $\delta_{q}$ is given by  
		\begin{equation} \label{eq:deltaq}
			{\delta_{q}}_{(x,y)}(a,b) = (g^{x-y} \cdot a, b), \quad \text{for all } x, y \in \mathbb{Z}_q \text{ and } a, b \in \mathbb{Z}_p.
		\end{equation} 
		\item {\bf  When $1< \gamma < q$:} In this case, $|\Lambda_{F_{\gamma}}'| = p^2$ since $|(F_{\gamma}, \circ)'| = p$.
	\end{enumerate}
	
	\noindent {\bf Skew left brace $G_{\mu}$:} In this case, $|\Lambda_{G_{\mu}}'| = p^2$ since $|(G_{\mu}, \circ)'| = p$ for $1 < \mu \leq q$. The lambda map is given by 
	\[ \lambda_{(n,m)}(s, t) = ( g^{-m}(g^t - 1) \cdot n + g^{m}\cdot s, t). \]
		Direct calculations shows that, for $1< \mu \leq q$, 
	\[
	\Lambda_{G_{\mu}} \cong (\mathbb{Z}_p \times \mathbb{Z}_p) \rtimes_{\delta_{\mu}} (\mathbb{Z}_q \times \mathbb{Z}_q),
	\]  
	where $\delta_{\mu}: \mathbb{Z}_q \times \mathbb{Z}_q \to \Aut(\mathbb{Z}_p \times \mathbb{Z}_p)$ given by  
	\begin{equation} \label{eq:delta of Amu}
		{\delta_{\mu}}_{(x,y)}(a,b) = \left( g^{x+(\mu - 1)y} \cdot a + (1 - g^x) g^{(\mu - 1)y} \cdot b, \, g^{(\mu - 1)y} \cdot b \right),
	\end{equation}
	for all \( x, y \in \mathbb{Z}_q \) and \( a, b \in \mathbb{Z}_p \).
	\par
	
	\begin{prop} \label{prop:Agammaisomorphism}
		Each group \( \Lambda_{F_{\gamma}} \) is isomorphic for \( 1 < \gamma < q \), and each group \( \Lambda_{G_{\mu}} \) is isomorphic for \( 1 < \mu \leq q \).
	\end{prop}
	\begin{proof}
		First, we handle the case of $\Lambda_{F_{\gamma}}$.  Observe that the action $\delta_{\gamma}$ (see Equation \eqref{eq:delta of Alambda}) can be decomposed as  
		$$ \delta_{\gamma} = \psi h_{\gamma}.$$ 
	where $\psi$ is given in Equation \eqref{eq:psi of D}
		and 
		\begin{align*}
 h_{\gamma} \in \Aut(\mathbb{Z}_q \times \mathbb{Z}_q), \quad h_{\gamma}(x,y) = (x+(\gamma-1)y, \gamma y),
		\end{align*}
	 for all \( x, y \in \mathbb{Z}_q \) and \( a, b \in \mathbb{Z}_p \). Thus, by Theorem \ref{semi isomorphism}, we obtain that
		\[
		\Lambda_{F_{\gamma}} \cong (\mathbb{Z}_p \times \mathbb{Z}_p) \rtimes_{\psi} (\mathbb{Z}_q \times \mathbb{Z}_q), \text{ for all } 1< \gamma < q.
		\]  
		The case of \( \Lambda_{G_{\mu}} \) is similar to that of \( F_{\gamma} \). Here, we take the maps  $\beta$ and $h_{\mu}$ in the following way:
		\begin{align*}
			\beta_{(x,y)}(a,b) & =  ((g^{x+y}\cdot a + (1-g^{x})g^{y}\cdot b, g^{y}\cdot b)) \\
			h_{\mu}(x,y) = & (x, (\mu - 1)y).
		\end{align*}
		It is easy to see that  $\delta_{\mu} = \beta h_{\mu}$ (see Equation \eqref{eq:delta of Amu} for  $\delta_{\mu}$). Then, by Theorem \ref{semi isomorphism}, we obtain that 
		\[
		\Lambda_{G_{\mu}} \cong (\mathbb{Z}_p \times \mathbb{Z}_p) \rtimes_{\beta} (\mathbb{Z}_q \times \mathbb{Z}_q),
		\]
		for all \( 1 < \mu \leq q \).
	\end{proof} 
	
	\begin{lemma} \label{lemma:ismorphism of A3}
		The groups $\Lambda_{ D}$, $\Lambda_{F_{\gamma}}$ for $1 < \gamma < q$, and $\Lambda_{G_{\mu}}$ for $1 < \mu \leq q$ are isomorphic.
	\end{lemma}
	\begin{proof}
		
	The groups \( \Lambda_{ D} \) and \( \Lambda_{F_{\gamma}} \) are  isomorphic, as both are isomorphic to  
	\[
	(\mathbb{Z}_p \times \mathbb{Z}_p) \rtimes_{\psi} (\mathbb{Z}_q \times \mathbb{Z}_q).\qquad \text{(see Equation \eqref{eq:psi of D})} 
	\]  
Using Proposition \ref{prop:Agammaisomorphism},  it is enough to  show that \( \Lambda_{F_{\gamma}} \) and \( \Lambda_{G_{\mu}} \) are isomorphic for some \( 1 < \gamma < q \) and \( 1 < \mu \leq q \).  For $(n,m), (s,t) \in G_2$, we have  
		\[
		\lambda^{\op}_{(n,m)}(s, t) = (g^{2m} \cdot s, m),
		\]  
		and we know that \( \Lambda_{G_{2}} \cong \Lambda_{G^{\op}_2} \). Direct calculations show that  
		\[
		\Lambda_{G^{\op}_{2}} \cong (\mathbb{Z}_p \times \mathbb{Z}_p) \rtimes_{\chi} (\mathbb{Z}_q \times \mathbb{Z}_q),
		\]  
		where  
		\[
		\chi_{(x,y)}(a,b) = (g^{x+2y} \cdot a, g^y \cdot b).
		\]  
		Define an automorphism \( f: \mathbb{Z}_q \times \mathbb{Z}_q \to \mathbb{Z}_q \times \mathbb{Z}_q \) by  
		\[
		f(x,y) = (x+2y, y).
		\]  
		Then \( \chi = \psi f \). From Proposition \ref{prop:Agammaisomorphism}, the groups \( \Lambda_{G_{\mu}} \), \( \Lambda_{F_{\gamma}} \), and \( \Lambda_{D} \) are isomorphic.
	\end{proof}
	
	\begin{lemma} \label{lemma:isomorphism of A4}
		The groups $ E$ and $F_q$  are isomorphic.
	\end{lemma}
	\begin{proof}
	Let $(n,m), (s,t) \in E$, the opposite of associated lambda map of $E$ is given by  
	\[
	\lambda^{\op}_{(n,m)}(s,t) = (g^m \cdot s, t).
	\]  
	Since \( \Lambda_E \cong \Lambda_{E^{\op}} \), through calculations, we get
	\[
	\Lambda_{E^{\op}} \cong (\mathbb{Z}_p \times \mathbb{Z}_p) \rtimes_{\eta} (\mathbb{Z}_q \times \mathbb{Z}_q),
	\]  
	where \( \eta_{(x,y)}(a,b) = (g^{x+y} \cdot a, b) \), for all \( x, y \in \mathbb{Z}_q \) and \( a, b \in \mathbb{Z}_p \).\\  
	Now, consider the automorphism \( f: \mathbb{Z}_q \times \mathbb{Z}_q \to \mathbb{Z}_q \times \mathbb{Z}_q \) given by \( f(x,y) = (x,-y) \). From Equation \eqref{eq:deltaq}, we have \( \delta_{q} = \eta f \), which implies that the groups \( \Lambda_{E} \) and \( \Lambda_{F_{q}} \) are isomorphic.	
	\end{proof}
	
	In the next subsection, we classify the irreducible representations for each skew brace of size $pq$. With the help of Theorem \ref{thm:mainresult} and the results obtained above, we then classify, up to isomorphism, the groups associated with the skew left braces in Theorem \ref{thm:classification}.
	
\subsection{Irreducible representations of skew left braces of order $pq$} \label{subsec:irrrepsskewbracepq}
Throughout this section, all the representations are over the field of complex numbers. We denote the set of all inequivalent irreducible representations of $G$ by $\Irr(G)$, and the set of the dimensions of irreducible representations of $G$ by $\cd(G)$.

In this section, we classify the irreducible representations of skew left braces of order $pq$ given in Theorem \ref{classification}. From Remark \ref{remark:equivalence of reps}, we achieve our goal by classifying the irreducible representations of the groups of order $p^2q^2$ associated with the skew left braces of size $pq$.
  Before we prove our results, we mention a few fundamental results in representation theory of finite groups that we use in our computations. 

\begin{lemma}  \label{lemma:represults}
	Let $G$ be a group. Then
	\begin{enumerate}
		\item [\rmfamily(i)] \textnormal{(\cite[Corollary 2.6]{IBook})} every irreducible representation is 1-dimensional if and only if $G$ is abelian.
		\item [\rmfamily(ii)] \textnormal{(\cite[Corollary 2.7]{IBook})} The number of irreducible representations of $G$ equals the number of conjugacy classes of $G$, and
		\[ |G| = \sum_{\rho \in \Irr(G)} \left( \dim(\rho) \right)^{2}. \]
		\item [\rmfamily(iii)] \textnormal{(\cite[Corollary 2.23]{IBook})} the number of 1-dimensional representations is equal to the index of the commutator subgroup in $G$.
		\item [\rmfamily(iv)] \textnormal{(\cite[Theorem 3.11]{IBook})} $\dim(\rho)$ divides the order of $G$, for every $\rho \in \Irr(G)$.
	\end{enumerate}
	
\end{lemma}

The following result provides a necessary and sufficient condition for a group to have $\{ 1, p \}$ as its set of the irreducible representation dimensions.
\begin{lemma} \textnormal{\cite[Theorem 12.11]{IBook}} \label{lemma:cdG=1,p}
	Let $G$ be a non-abelian group. Then $\cd(G) = \{ 1, p \}$ if and only if one of the following holds.
	\begin{enumerate}
		\item [\rmfamily(i)] There exists an abelian normal subgroup $A$ of $G$ of index $p$.
		\item [\rmfamily(ii)] $|G:Z(G)| = p^3$.
	\end{enumerate} 
\end{lemma}

For solvable groups, we have the following result.
\begin{lemma} \textnormal{\cite[Corollary 12.34]{IBook}} \label{lemma:pNotDivideCharacterdegrees}
	Let $G$ be a solvable group. Then $G$ has a normal abelian Sylow $p$-subgroup of $G$ if and only if every element of $\cd(G)$ is relatively prime to $p$.
\end{lemma}
\noindent Observe that if $G$ is a group of order $p^2q^2$ ($p>q$), then Sylow $p$-subgroup is unique in $G$. Then, by Lemma \ref{lemma:pNotDivideCharacterdegrees}, every element of $\cd(G)$ is relatively prime to $p$. Hence, by Lemma \ref{lemma:represults}, the possibility for the dimension of an irreducible representation of a group of order $p^2q^2$ is 1, $q$ or $q^2$. We shall be using this information to prove our results.

In Lemma \ref{lemma:action not faithful}, we compute the irreducible representations for a class of groups of order $p^2q^2$.

\begin{lemma} \label{lemma:action not faithful}
Let \( G \) be a non-abelian group of order \( p^2q^2 \). Then \( \cd(G) = \{1, q, q^2\} \) if and only if the action of the Sylow \( q \)-subgroup of \( G \) on the Sylow \( p \)-subgroup of \( G \) is faithful.
Moreover, if the action is not faithful, then there are exactly \( p^{2-i}q^2 \) one-dimensional representations and \( (p^2 - p^{2-i}) \) irreducible representations of dimension \( q \), where \( |G'| = p^i \) for some \( i \) with \( 1 \leq i \leq 2 \).
\end{lemma}
\begin{proof}
	Let \( G = P \rtimes_{\psi} Q \) such that \( \ker(\psi) \) is nontrivial. Then \( P \times \ker(\psi) \) is a subgroup of \( G \). Since \( \ker(\psi) \leq Q \), we must have \( |\ker(\psi)| = q \) or \( q^2 \). However, if \( |\ker(\psi)| = p^2 \), then \( G \) would contain an abelian subgroup of order \( p^2q^2 \), contradicting the non-abelian nature of \( G \). Hence, we conclude that \( |\ker(\psi)| = q \), and consequently, \( P \times \ker(\psi) \) is an abelian subgroup of order \( p^2q \) in \( G \).
	
	Moreover, \( P \times \ker(\psi) \) is normal in \( G \) since its index in \( G \) is \( q \) and \( p > q \). By Lemma \ref{lemma:cdG=1,p}, we have \( \cd(G) = \{ 1, q \} \). Let \( |G'| = p^i \) for some \( i \) with \( 1 \leq i \leq 2 \). From Lemma \ref{lemma:represults}, the number of 1-dimensional representations of \( G \) is given by \( |G/G'| \), and we have the relation
	\[
	|G| = p^2q^2 = |G/G'| + r \cdot q^2,
	\]
	where \( r \) denotes the number of irreducible representations of \( G \) of dimension \( q \). Therefore, there are exactly \( p^{2-i}q^2 \) one-dimensional representations and \( (p^2 - p^{2-i}) \) irreducible representations of dimension \( q \).
	
\end{proof}

Now, we prove our main result.

\subsection*{Proof of Theorem \ref{thm:mainresult}}
If $p \not\equiv 1 \bmod q$, then there is only the trivial skew brace, say $A$. In this case, $\Lambda_{A}$ is an abelian group of order $p^2 q^2$. By Lemma \ref{lemma:represults}, $\Lambda_A$ does not have any nonlinear irreducible representations and there are $p^2 q^2$ many representations of dimension 1.

Now onwards, $p \equiv 1 \bmod q$ unless stated otherwise. We deal with each skew left brace separately.\\

\noindent \textbf{Representations of $ B$:} Clearly, $\Lambda_{ B}$ is an abelian group of order $p^2 q^2$. Hence, $\Lambda_B$ does not have any nonlinear irreducible representations and there are $p^2 q^2$ many representations of dimension 1.\\

\noindent \textbf{Representations of $ C$:} By Equation \eqref{eq:lambda of A2},  we have
\[
\lambda_{(n,m)}(s, t) = (g^{m} \cdot s, t), \text{ for all } n, s \in \mathbb{Z}_p \text{ and } m, t \in \mathbb{Z}_q.
\] 
Direct calculation shows that the subgroup $\mathbb{Z}_{p} \times \{ 0 \}$ is contained in the kernel of the lambda map of $ C$. Hence, the subgroup $( C, \cdot)  \rtimes_{\lambda} (\mathbb{Z}_{p} \times \{ 0 \})$ is an abelian subgroup of order $p^2q$.
By Lemma \ref{lemma:action not faithful}, $\Lambda_{C}$ has $q^2$ many representations of dimension 1 and $p^2 - 1$ many irreducible representations of dimension $q$.\\

\noindent {\bf Representations of $ D$:}  In this case, $\Lambda_{ D} = \mathbb{Z}_{p} \rtimes \mathbb{Z}_{q} \times \mathbb{Z}_{p} \rtimes \mathbb{Z}_{q}$. Let us denote the group $\mathbb{Z}_{p} \rtimes \mathbb{Z}_{q}$ by $H$ so that $|H'| = p$ and $\Lambda_{D} = H\times H$. Since $H$ has an abelian normal subgroup of index $q$, by Lemma \ref{lemma:cdG=1,p}, we get $\cd(H) = \{ 1, q \}$. Hence, from Lemma \ref{lemma:represults}, there are exactly $q$ many 1-dimensional representations of $H$ and $(p-1)/q$ many $q$-dimensional irreducible representations of $H$. By \cite[Theorem 4.21]{IBook}, for each $\rho \in \Irr(H\times H)$, we have
\[ \dim(\rho) = \dim(\rho_1)\dim(\rho_2), \text{ for some } \rho_1, \rho_2 \in \Irr(H). \]
Therefore,
$\cd(\Lambda_{D}) = \{ 1, q, q^2 \}$ and we have $q^2$ many 1-dimensional representations, $2(p-1)$ many $q$-dimensional irreducible representations and $(p-1)^2/q^2$ many $q^2$-dimensional irreducible representations of $\Lambda_{D}$.  \\

\noindent {\bf Representations of $ E$:} In this case, the subgroup  $(\mathbb{Z}_p \times \mathbb{Z}_p) \rtimes_{\xi} (\{0 \} \times \mathbb{Z}_q) $ is abelian and  has index $q$ in
$\Lambda_{E}$, and $\{0\} \times \mathbb{Z}_q$ is contained in the kernel of $\xi$ (see Equation \eqref{eq:xi of E} for the definition of $\xi$).
By Lemma \ref{lemma:action not faithful}, $\Lambda_{E}$ has $pq^2$ many representations of dimension 1 and $p(p- 1)$ many irreducible representations of dimension $q$. \\

\noindent {\bf  Representations of $F_{\gamma}$:} We have the following two cases:
\begin{enumerate}
	\item [\rmfamily(i):] By Lemma \ref{lemma:ismorphism of A3}, when $1 < \gamma < q$, $\Lambda_{F_{\gamma}}$ has  $q^2$ many 1-dimensional representations, $2(p-1)$ many $q$-dimensional irreducible representations and $(p-1)^2/q^2$ many $q^2$-dimensional irreducible representations.
	\item [\rmfamily(ii):] By Lemma \ref{lemma:isomorphism of A4}, $\Lambda_{F_q}$ has $pq^2$ many representations of dimension 1 and $p(p- 1)$ many irreducible representations of dimension $q$.
\end{enumerate}
\vspace{5mm}

\noindent {\bf  Representations of $G_{\mu}$:} By Lemma \ref{lemma:ismorphism of A3}, when $1 < \gamma \leq q$, $\Lambda_{G_{\mu}}$ has  $q^2$ many 1-dimensional representations, $2(p-1)$ many $q$-dimensional irreducible representations and $(p-1)^2/q^2$ many $q^2$-dimensional irreducible representations. \qed \\

With the help of Theorem \ref{thm:mainresult}, we classify up to isomorphism the groups associated with the skew left braces of size $pq$.

	\begin{thm} \label{thm:classification}
	Let $A$ be a skew left brace of size $pq$. Then $\Lambda_A$ is isomorphic to one of the following groups 
	\begin{enumerate}
		\item $(\mathbb{Z}_{pq} \times \mathbb{Z}_{pq})$,
		\item $(\mathbb{Z}_{p} \times \mathbb{Z}_{p}) \rtimes_{\theta} (\mathbb{Z}_{q} \times \mathbb{Z}_{q})$,
\item $(\mathbb{Z}_{p} \times \mathbb{Z}_{p}) \rtimes_{\psi} (\mathbb{Z}_{q} \times \mathbb{Z}_{q})$,
\item $(\mathbb{Z}_p \times \mathbb{Z}_p) \rtimes_{\eta} (\mathbb{Z}_q \times \mathbb{Z}_q)$,
	\end{enumerate}
	where $\theta$, $\psi$ and $\xi$ are given in Equation \eqref{eq:theta of C}, \eqref{eq:psi of D} and \eqref{eq:xi of E}, respectively.
\end{thm}
\begin{proof} The proof of part (i) follows from Lemma \ref{lemma:ismorphism of A3}, and that of part (ii) follows from \ref{lemma:isomorphism of A4}. Further, from Theorem \ref{thm:mainresult}, the number of irreducible representations of $A, B, C$ and $D$ are different, which completes the proof.
\end{proof}

We conclude this article with the following examples in which we obtain a $q$-dimensional irreducible representation of $C$, and hence of $\Lambda_{C}$.
\begin{example} \label{example: q dim rep skew brace}
	Let us consider the skew left brace $C$, where $p=3$ and $q=2$. Then $(C, +)$ is the cyclic group of order 6 and $(C, \circ)$ is the dihedral group of order 6, denoted by $D_6$. Note that $D_{6} = \langle (1, 0), (0, 1) \rangle$.
	 Let $\beta$ be a 2-dimensional representation of $(C, +) = \mathbb{Z}_{6}$ defined by
	\[ \beta(1,1) = 
	\left(
	\begin{array}{cc}
		\omega_3 & 0 \\
		0 & \omega_3^{2}
	\end{array}
	\right),
	 \]
	 and $\rho$ be the 2-dimensional irreducible representation of $(C, \circ) = D_6$ given by
	 \[ \rho(1,0) = 
	 \left(
	 \begin{array}{cc}
	 	\omega_3 & 0 \\
	 	0 & \omega_3^{2}
	 \end{array}
	 \right)
	  \text{ and }
	   \rho(0,1) = 
	 \left(
	 \begin{array}{cc}
	 	0 & 1 \\
	 	1 & 0
	 \end{array}
	 \right).
	  \]
	  Here, $\omega_3$ is a primitive cube root of unity and  $\lambda_{(n,m)}^{\op}(s,t) = (2^m \cdot s, t)$. Through routine computations, it is easy to verify that Equation \eqref{relation} holds. Therefore, $(V, \beta, \rho)$ is a 2-dimensional representation of $C$. Since $\rho$ is irreducible, $(V, \beta, \rho)$ is irreducible as well. By \cite[Remark 2.2]{KT24}, 
	  \[ \phi_{(\beta, \rho)}: \Lambda_{ C} \to \GL(V), \quad \phi_{(\beta, \rho)}(a,b) = \beta(a)\rho(b) \]
	  is a representation of $\Lambda_{C}$ given by
	  \begin{align*}
	  	((s,t),(n,0)) & \mapsto \left(
	  	\begin{array}{cc}
	  		\omega_3^{l+n} & 0 \\
	  		0 & \omega_3^{2(l+n)}
	  	\end{array}
	  	\right), \quad ((s,t),(0,1)) \mapsto \left(
	  	\begin{array}{cc}
	  	0 &	\omega_3^{l} \\
	  	 \omega_3^{2l} & 0
	  	\end{array}
	  	\right),\\
	  	((s,t),(1,1)) & \mapsto \left(
	  	\begin{array}{cc}
	  		0 &	\omega_3^{l+1} \\
	  		\omega_3^{2(l+1)} & 0
	  	\end{array}
	  	\right), \quad
	  	((s,t),(2,1)) \mapsto \left(
	  	\begin{array}{cc}
	  		0 &	\omega_3^{2+l} \\
	  		\omega_3^{2l+1} & 0
	  	\end{array}
	  	\right),
	  \end{align*}
	  where $s,n \in \mathbb{Z}_{3}$, $t\in \mathbb{Z}_{2}$ and $l\in \mathbb{N}$ such that $l(1,1) = (s,t)$. It is easy to check that $\phi_{(\beta, \rho)}$ is an irreducible 2-dimensional representation of $\Lambda_{ C}$.
\end{example}
There is another way, through group-theoretic arguments, to obtain a $q$-dimensional irreducible representation of $\Lambda_{C}$. In this, we use \cite[Proposition 25]{SerreBook}.
\begin{example} \label{example:q dim rep group theory}
From Equation \eqref{eq:theta of C}, $\Lambda_{C} \cong P \rtimes_{\theta} Q$, where $P = \mathbb{Z}_{p} \times \{ 0 \} \times \mathbb{Z}_{p} \times \{ 0 \} \cong \mathbb{Z}_{p} \times \mathbb{Z}_{p}$ and $Q = \{ 0 \} \times \mathbb{Z}_{q} \times \{ 0 \} \times \mathbb{Z}_{q} \cong \mathbb{Z}_{q} \times \mathbb{Z}_{q}$.
Let $\chi$ be a 1-dimensional representation of $P$ such that $\chi = \chi_1 \chi_2$, where $\chi_1, \chi_2 \in \Irr(\mathbb{Z}_{p})$ given by $\chi_1(a,0) = 1$ and $\chi_2(b,0) = \omega_p$, where $\omega_p$ is a $p^{th}$ root of unity and $(a, b)$ is a generator of $P$. Let $Q_1$ be the subgroup of $Q$ consisting of those elements $h = (x,y) := ((0,x),(0,y))$ such that $\chi(hkh^{-1}) = \chi(k)$ for all $k \in P$. Since the action of $Q$ on $P$ is given by 
\[ \theta_{(x,y)}(a,b) = (g^{y}\cdot a, g^{y}\cdot b), \]
we obtain 
\[ \chi(((0,x),(0,y))((a,0),(b,0))((0,x),(0,y))^{-1}) = \chi((g^{y}\cdot a, g^{y}\cdot b)) = \omega_{p}^{g^{y}}. \]
Now, $\omega_{p}^{g^{y}} = \chi(a,b) = \chi_1(a,0)\chi_2(b,0) = \omega_p$ implies that $g^{y} \equiv 1 \bmod p$ and $y \equiv 0 \bmod q$. Thus, $(x,y) \notin Q_1$ if $y\neq 0$, and hence, $|Q_1| = q$.\\
Take $G_1 = PQ_1$ so that $|G_1| = p^2 q$. Let $\rho$ be an irreducible representation of $Q_1$ and $\pi$ be the canonical projection of $G_1$ on $Q_1$. Then $\tilde{\rho}:\rho \circ \pi$ is an irreducible representation of $G_1$. The tensor product $\chi \otimes \tilde{\rho}$ is again an irreducible representation of $G_1$, where 
\[ \chi \otimes \tilde{\rho}(t) = \chi(t)\tilde{\rho}(t), \quad \chi(t), \tilde{\rho}(t) \in {\mathbb{C}}^{*}. \]
Let $\theta_{\chi,\rho}$ be the corresponding induced representation of $G$. Then 
\[ \dim\left( \theta_{\chi,\rho} \right) = |\Lambda_{A_{\gamma}} : G_1| \dim\left( \chi \otimes \tilde{\rho} \right) = q \cdot 1 = q. \]
From \cite[Proposition 25]{SerreBook}, $\theta_{\chi,\rho}$ is an irreducible $q$-dimensional representation of $\Lambda_{C}$.
\end{example}
In fact, as is stated in \cite[Proposition 25]{SerreBook}, all the irreducible representations of the group can be obtained following the process given in the example above. \\

\section*{Acknowledgments}
 Nishant Rathee acknowledges the support of the NBHM Postdoctoral Fellowship under grant number 0204/16(1)(1)/2024/R\&D-2/10821.  The author also thanks ISI Delhi for providing a conducive research environment. Ayush Udeep thanks IISER Mohali for an institute post doctoral fellowship.


\begin{thebibliography}{99}
		
		\bibitem{AB2020} E. Acri and M. Bonatto, Skew braces of size $pq$, Comm. Algebra 48(05) (2020), 1872--1881.
		
	\bibitem{DB18}	 D.~Bachiller, \textit{Extensions, matched products, and simple braces}, J. Pure Appl. Algebra 222 (2018), 1670--1691.

\bibitem{BR21} A.~ Ballester-Bolinches and R.~ Esteban-Romero, \textit{Triply Factorised Groups and the Structure of Skew Left Braces}, Commun. Math. Stat. 10 (2022), 353--370.

\bibitem{CSV18} F.~ Cedo, A.~Smoktunowicz, L.~ Vendramin, \textit{Skew left braces of nilpotent type},  Proc. London Math. Soc. 118 (2018), 1367--1392.

		\bibitem{GV17} L.~ Guarnieri and L.~ Vendramin, \textit{Skew braces and the Yang-Baxter equation}, Math. Comp. 86 (2017), no. 307, 2519--2534. 
		
		\bibitem{IBook} 	I. M. Isaacs, \textit{Character theory of finite groups}, Corrected reprint of the 1976 original, AMS Chelsea Publishing, Providence, RI, 2006.  
	
		\bibitem{KT24}Y. Kozakai and C. Tsang, \textit{Representation theory of skew braces}, Int. J. Group Theory 14(03) (2025), 149--164.
%
		\bibitem{LV23} T. Letourmy and L. Vendramin, \textit{Isoclinism of skew braces}, Bull. Lond. Math. Soc. 55(06) (2023), 2891--2906.
%
		\bibitem{LV24} T. Letourmy and L. Vendramin, \textit{Schur covers of skew braces}, J. Algebra 644 (2024), 609--654.	
		

		\bibitem{WR07} W. Rump,  \textit{Braces, radical rings,  and the quantum Yang Baxter equation}, J. Algebra 307 (2007), 153--170.
		
		
		\bibitem{RSU24} N.~Rathee, M.~ Singh and  A.~ Udeep, \textit{Representations of  skew braces}, http://arxiv.org/abs/2408.03766 (2024).
		
		\bibitem{SerreBook} J. Serre, Linear Representations of Finite Groups, Vol. 42. New York, Springer, 1977.
		
		\bibitem{AV18} A.~ Smoktunowicz and L.~ Vendramin, \textit{On skew braces (with an appendix by
		N. Byott and L. Vendramin)}, J. Comb. Algebra 2(01) (2018), 47--86. 

		
	\end{thebibliography}
\end{document}